\documentclass[english]{amsart}
\usepackage[T1]{fontenc}
\usepackage[latin9]{inputenc}
\usepackage{mathrsfs}
\usepackage{amstext}
\usepackage{amsthm}
\usepackage{amssymb}

\makeatletter
\numberwithin{equation}{section}
\numberwithin{figure}{section}
\theoremstyle{plain}
\newtheorem{thm}{\protect\theoremname}
\theoremstyle{definition}
\newtheorem{defn}[thm]{\protect\definitionname}
\theoremstyle{plain}
\newtheorem{lem}[thm]{\protect\lemmaname}

\makeatother

\usepackage{babel}
\providecommand{\definitionname}{Definition}
\providecommand{\lemmaname}{Lemma}
\providecommand{\theoremname}{Theorem}

\begin{document}
\title{{\large{}a combinatorial approach to central set theorem}}
\author{sayan goswami}
\address{{\large{}Department of Mathematics, University of Kalyani, Kalyani-741235,
Nadia, West Bengal, India.}}
\email{{\large{}sayan92m@gmail.com}}
\begin{abstract}
{\large{}H. Furstenberg introduced the notion of central set in terms
of topological dynamics and established the central set theorem. The
essence of central set theorem is that it is the simultaneous extension
of van der Waerden's theorem and Hindman's theorem. Later V. Bergelson
and N. Hindman established a connection between central sets and the
algebra of Stone-\v{C}ech compactification and proved that the central
sets are the member of minimal idempotent ultrafilters of discrete
semigroup. In some subsequent papers the central set theorem has been
studied more deeply and some generalizations has been established.
Those works use the techniques of Stone-\v{C}ech compactification
of discrete semigroup. In this work we will prove central set theorem
via combinatorial approach using the combinatorial characterization
of central set established by N. Hindman, A. Malkeki and D. Strauss.
Though we will use the combinatorial approach but the technique of
the proof is similar to the proof of D. De, N. Hindman, D. Strauss.}{\large\par}
\end{abstract}

\maketitle

\section{\textbf{\large{}introduction}}

{\large{}A subset $S$ of $\mathbb{Z}$ is called syndetic if there
exists $r\in\mathbb{N}$ such that $\bigcup_{i=1}^{r}\left(S-i\right)=\mathbb{Z}$
and it is called thick if it contains arbitrary long intervals in
it. Sets which can be expressed as intersection of thick and syndetic
sets are called piecewise syndetic sets.}{\large\par}

{\large{}For a general semigroup $\left(S,\cdot\right)$, a set $A\subseteq S$
is said to be syndetic in $\left(S,\cdot\right)$, if there exists
a finite nonempty set $F\subseteq S$ such that $\bigcup_{t\in F}t^{-1}A=S$
where $t^{-1}A=\left\{ s\in S:t\cdot s\in A\right\} $. A set $A\subseteq S$
is said to be thick if for every finite nonempty set $E\subseteq S$,
there exists an element $x\in S$ such that $E\cdot x\subseteq A$.
A set $A\subseteq S$ is said to be piecewise syndetic set if there
exist a finite nonempty set $F\subseteq S$ such that $\bigcup_{t\in F}t^{-1}A$
is thick in $S$. It can be proved that a piecewise syndetic set is
the intersection of a thick set and a syndetic set.}{\large\par}

{\large{}Central set was defined in terms of topological dynamics
by H. Furstenberg \cite{key-3} and it's equivalent algebraic characterization
was established by V. Bergelson and N. Hindman in \cite{key-8}. H.
Furstenberg first established the central set theorem and his proof
was in terms of topological dynamics. Later this theorem was generalized
and all those proofs uses the techniques of algebraic structure of
the Stone-\v{C}ech compactification of discrete semigroup. One can
see \cite{key-1.1} for the details history of central set theorem.
Recently we came to know from \cite{key-8} that there is no combinatorial
approach to central set theorem is known. So in this work we tried
to give a combinatorial approach to central set theorem. To do this
we will use the combinatorial characterization of central sets. }\\
{\large{}Central sets has an combinatorial characterization which
will be needed for our purpose, stated below.}{\large\par}
\begin{thm}
{\large{}\label{Charac. central set} \cite[Theorem 3.8]{key-5} For
a countable semigroup $\left(S,.\right)$, $A\subseteq S$ is said
to be central iff there is a decreasing sequence $\langle C_{n}\rangle_{n=1}^{\infty}$
of subsets of $A$ such that,}{\large\par}
\begin{enumerate}
\item {\large{}\label{1.1} for each $n\in\mathbb{N}$ and each $x\in C_{n}$,
there exists $m\in\mathbb{N}$ with $C_{m}\subseteq x^{-1}C_{n}$
and}{\large\par}
\item {\large{}\label{1.2} $C_{n}$ is collectionwise piecewise syndetic
$\forall\,n\in\mathbb{N}$.}{\large\par}
\end{enumerate}
\end{thm}

{\large{}The following famous central set theorem is due to H. Furstenberg.}{\large\par}
\begin{thm}
{\large{}\label{ (Furstenberg-Central Sets Theorem )} \cite{key-3}
Let $A$ be a central subset of $\mathbb{N}$, let $k\in\mathbb{N}$
and for each $i\in\left\{ 1,2,...,k\right\} $, let $\left\langle y_{l,n}\right\rangle _{n=1}^{\infty}$
be a sequence in $\mathbb{Z}$. There exist sequences $\left\langle a_{n}\right\rangle _{n=1}^{\infty}$
in $\mathbb{N}$ and $\left\langle H_{n}\right\rangle _{n=1}^{\infty}$
in $P_{f}\left(\mathbb{N}\right)$ such that }{\large\par}

{\large{}(1) for each $n,\:\max H_{n}<\min H_{n+1}$ and }{\large\par}

{\large{}(2) for each $i\in\left\{ 1,2,...,k\right\} $ and each $F\in P_{f}\left(\mathbb{N}\right)$,
\[
\sum_{n\in F}\left(a_{n}+\sum_{t\in H_{n}}y_{i,t}\right)\in A.
\]
}{\large\par}
\end{thm}

{\large{}The following theorem is called central set theorem for arbitrary
semigroup.}{\large\par}
\begin{thm}
{\large{}\label{Theorem 1.8} Let $\left(S,+\right)$ be a countable
commutative semigroup, let $A$ be a central set in $S$, and for
each $l\in\mathbb{N}$, let $\left\langle y_{l,n}\right\rangle _{n=1}^{\infty}$
be a sequence in $S$. There exist a sequence $\left\langle a_{n}\right\rangle _{n=1}^{\infty}$
in $S$ and a sequence $\left\langle H_{n}\right\rangle _{n=1}^{\infty}$
in $\mathcal{P}_{f}\left(\mathbb{N}\right)$ such that $\max H_{n}<\min H_{n+1}$
for each $n\in\mathbb{N}$ and such that for each $f\in\Phi$, 
\[
FS\left(\left\langle a_{n}+\sum_{t\in H_{n}}y_{f\left(n\right),t}\right\rangle _{n=1}^{\infty}\right)\subseteq A,
\]
 where $\Phi$ is the set of all functions $f:\mathbb{N\rightarrow N}$
for which $f\left(n\right)\leq n$ for all $n\in\mathfrak{\mathbb{N}}$.}{\large\par}
\end{thm}

{\large{}The following central set theorem is for general commutative
semigroup:}{\large\par}
\begin{thm}
{\large{}\label{( Stronger Central Sets Theorem )}\cite[Theorem 14.8.4, page 337]{key-7}
Let $\left(S,+\right)$ be a commutative semigroup. Let $C$ be a
central subset of $S$. Then there exist functions $\alpha:\mathcal{P}_{f}\left(S^{\mathbb{N}}\right)\rightarrow\mathbb{N}$
such that }{\large\par}

{\large{}1) let $F,G\in\mathcal{P}_{f}\left(S^{\mathbb{N}}\right)$
and $F\subsetneq G$, then $\max H\left(F\right)<\min H\left(G\right)$,}{\large\par}

{\large{}2) whenever $r\in\mathbb{N}$, $G_{1},G_{2},...,G_{r}\in\mathcal{P}_{f}\left(^{\mathbb{N}}S\right)$
such that $G_{1}\subsetneq G_{2}\subsetneq.....\subsetneq G_{r}$
and for each $i\in\left\{ 1,2,....,r\right\} $, $f_{i}\in G_{i}$
one has 
\[
\sum_{i=1}^{r}\left(\alpha\left(G_{i}\right)+\sum_{t\in H\left(G_{i}\right)}f_{i}\left(t\right)\right)\in C.
\]
}{\large\par}
\end{thm}

{\large{}Various stronger non-commutative version of the above central
set theorem can be found in \cite{key-2}.}{\large\par}

{\large{}There is an important set, which is intimately related to
central set theorem is known as $J$-set. In commutative semigroup
it is defined as,}{\large\par}
\begin{defn}
{\large{}\label{definition 1.10} \cite{key-6} Let $(S,+)$ be a
commutative semigroup and let $A\subseteq S$ is said to be a $J$-set
iff whenever $F\in\mathcal{P}_{f}\left(S^{\mathbb{N}}\right)$, there
exist $a\in S$ and $H\in\mathcal{P}_{f}(\mathbb{N})$ such that for
each $f\in F,$ $a+\sum_{t\in H}f(t)\in A$.}{\large\par}
\end{defn}

{\large{}In non-commutative case the situation is little different.
Here the analogous notion of $J$-sets are defined as,}{\large\par}
\begin{defn}
{\large{}\label{definition 1.11}\cite[Definition 14.14.1, page 342]{key-7}
Let $\left(S,\cdot\right)$ is a semigroup.}{\large\par}
\begin{enumerate}
\item {\large{}$\mathcal{T}=S^{\mathbb{N}}$}{\large\par}
\item {\large{}For $m\in\mathbb{N}$, $\mathcal{J}_{m}=$$\left\{ \begin{array}{cc}
\left(t\left(1\right),t\left(2\right),\ldots,t\left(m\right)\right)\in\mathbb{N}^{m} & :\\
t\left(1\right)<t\left(2\right)<\ldots<t\left(m\right)
\end{array}\right\} $}{\large\par}
\item {\large{}Given $m\in\mathbb{N}$, $a\in S^{m+1}$, $t\in\mathcal{J}_{m}$
and $f\in\mathcal{T}$, 
\[
x\left(m,a,t,f\right)=\left(\prod_{j=1}^{m}\left(a\left(j\right)\cdot f\left(t\left(j\right)\right)\right)\right)\cdot a\left(m+1\right)
\]
}{\large\par}
\item {\large{}$A\subseteq S$ is called a $J-set$ iff for each $F\in\mathcal{P}_{f}\left(\mathcal{T}\right)$,
there exists $m\in\mathbb{N}$, $a\in S^{m+1}$, $t\in\mathcal{J}_{m}$
such that, for each $f\in\mathcal{T}$,}\\
{\large{}
\[
x\left(m,a,t,f\right)\in A.
\]
}{\large\par}
\end{enumerate}
{\large{}We will use the Hales-Jewett theorem in our proof. The following
is a brief introduction to that theorem.}{\large\par}
\end{defn}

{\large{}Conventionally $[t]$ denotes the set $\{1,2,\ldots,t\}$
and words of length $N$ over the alphabet $[t]$ are the elements
of $[t]^{N}.$A variable word is a word over $[t]\cup\{*\}$ in where
$*$ occurs at least once and $*$ denotes the variable. A combinatorial
line is denoted by $L_{\tau}=\{\tau(1),\tau(2),\ldots,\tau(t)\}$
where $\tau(*)$ is a variable word and $L_{\tau}$ is obtained by
replacing the variable $*$ by $1,2,\ldots.t.$}{\large\par}

{\large{}The following theorem is due to Hales-Jewett.}{\large\par}
\begin{thm}
{\large{}\cite{key-4} For all values $t,r\in\mathbb{N}$, there exists
a number $HJ(r,t)$ such that, if $N\geq HJ(r,t)$ and $[t]^{N}$
is $r$ colored then there will exists a monochromatic combinatorial
line.}{\large\par}
\end{thm}

{\large{}The following is a strong form of Hales-Jewett theorem,}{\large\par}
\begin{lem}
{\large{}\cite[Lemma 14.8.1, page 335]{key-6} Let $A$ be a finite
nonempty alphabet and let $r\in\mathbb{N}$. There is some $m\in\mathbb{N}$
such that whenever the length $m$ words over $A$ is r-colored, there
is a variable word $\omega(v)$ such that $\omega(v)$ begins and
ends with a constant, $\omega(v)$ has no occurrences of $v$, and
$\{\omega(a):a\in A\}$ is monochromatic.}{\large\par}
\end{lem}

{\large{}In section 2, we will give a combinatorial proof of central
set theorem for commutative semigroup and in section 3, this technique
will be extended to non commutative semigroup.}{\large\par}

\section{\textbf{\large{}central set theorem in commutative semigroup}}

{\large{}In this section we will work with the commutative semigroup
and use the notation $\left(S,+\right)$ to denote this.}{\large\par}
\begin{lem}
{\large{}\label{commutative J set} Every piecewise syndetic set is
a J- set}{\large\par}
\end{lem}

\begin{proof}[Proof of stronger central set theorem]
{\large{}Let $\left(S,+\right)$ be a commutative semigroup. Let $A\subseteq S$
be a piecewise syndetic set. Then there exists a finite set $E$ }\\
{\large{}such that $\cup_{t\in E}-t+A$ is thick. Let $\mid E\mid=r$.
Let $F\in\mathcal{P}_{f}\left(^{\mathbb{N}}S\right)$ be given. Let
$\mid F\mid=n$.}\\
{\large{}Enumerate the set $F$ as $F=\left\{ f_{1},f_{2},\ldots,f_{n}\right\} $
}\\
{\large{}Then take the Hales-Jewett number $N=N\left(r,n\right)$.
Consider the set $G=[n]^{N}$and consider the map}\\
{\large{}$g:G\rightarrow S$ defined by $g\left(a_{1},a_{2},\ldots,a_{N}\right)=\sum f_{a_{i}}\left(i\right)$.}\\
{\large{}Now as $g\left(G\right)$ is a finite set, there exists an
element $b\in S$ such that $b+g\left(G\right)\subset\cup_{t\in E}-t+A$.}\\
{\large{}Now induce an $r$-color $\chi$ of $g\left(G\right)$ as
$\chi\left(a\right)=i$ iff $b+g\left(a\right)\in-t_{i}+A$ for minimum
$1\leq i\leq r$.}\\
{\large{}Then there is a monochromatic combinatorial line in $G$
and this correspond to a configuration 
\[
b+a+\sum_{t\in H}f_{i}\left(t\right)\in A\,for\,allf_{i}\in F
\]
i.e., there is an element $s\in S$ such that $s+\sum_{t\in H}f_{i}\left(t\right)\in A\,$
for all $f_{i}\in F$.}{\large\par}

{\large{} Let $A$ be central set in $\left(S,+\right)$ and then
from theorem \ref{Charac. central set} there exists a chain of piecewise
syndetic set 
\[
A\supseteq A_{1}\supseteq A_{2}\supseteq\cdots\supseteq A_{n}\supseteq\cdots
\]
satisfying property \ref{1.1}. Let us fixed any $N\in\mathbb{N}$
and take the pieecewise syndetic set $A_{N}$.}\\
{\large{}We define $\alpha\left(G\right)\in S$ and $H\left(G\right)\in\mathcal{P}_{f}\left(\mathbb{N}\right)$
for $F\in\mathcal{P}_{f}\left(S^{\mathbb{N}}\right)$ by induction
on $\mid F\mid$ satiesfying the following inductive hypothesis:}{\large\par}

{\large{}1) for $F,G\in\mathcal{P}_{f}\left(S^{\mathbb{N}}\right)$
and $\emptyset\neq G\subsetneq F$, then $\max H\left(G\right)<\min H\left(F\right)$,}{\large\par}

{\large{}2) whenever $n\in\mathbb{N}$, $G_{1},G_{2},...,G_{n}\in\mathcal{P}_{f}\left(^{\mathbb{N}}S\right)$
such that $G_{1}\subsetneq G_{2}\subsetneq.....\subsetneq G_{n}=F$
and $\langle f_{i}\rangle_{i=1}^{n}\in\times_{i=1}^{n}G_{i}$ one
has 
\[
\sum_{i=1}^{n}\left(\alpha\left(G_{i}\right)+\sum_{t\in H\left(G_{i}\right)}f_{i}\left(t\right)\right)\in A_{N}
\]
}{\large\par}

{\large{}Let $F=\left\{ f\right\} $ then as from lemma \ref{commutative J set}
$A_{N}$ is piecewise syndetic set there is $a\in S$ and $L\in\mathcal{P}_{f}\left(\mathbb{N}\right)$
such that}\\
{\large{}$a+\sum_{t\in L}f\left(t\right)\in A_{N}$. Now define $\alpha\left(\left\{ f\right\} \right)=a$
and $H\left(\left\{ f\right\} \right)=L$. }\\
{\large{}Now assume $\mid F\mid>1$ and $\alpha\left(G\right)\in S$
and $H\left(G\right)\in\mathcal{P}_{f}\left(\mathbb{N}\right)$ have
been defined for all proper subsets $G$ of $F$. }\\
{\large{}Let $K=\bigcup\left\{ H\left(G\right):\emptyset\neq G\subsetneq F\right\} $and
let $m=\max K$.}{\large\par}

{\large{}Let,
\begin{align*}
M & =\{\sum_{i=1}^{n}\left(\alpha\left(G_{i}\right)+\sum_{t\in H\left(G_{i}\right)}f_{i}\left(t\right)\right):n\in\mathbb{N}\,and\\
 & G_{1}\subsetneq G_{2}\subsetneq.....\subsetneq G_{n}\subsetneq Fand\langle f_{i}\rangle_{i=1}^{n}\in\times_{i=1}^{n}G_{i}\}
\end{align*}
As $M$ is a finite set and $M\subset A_{N}$, let $A_{N}\cap_{x\in M}\left(-x+A_{N}\right)\supseteq A_{P}$
for some $P\geq N$. Let $F$ then as $A_{P}$ is piecewise syndetic
set, it is a J-set. So, by \cite[Lemma 14.8.2]{key-6} there exists
$a\in S$ and $L\in\mathcal{P}_{f}\left(\mathbb{N}\right)$ such that
$\min L>m$ and $a+\sum_{t\in L}f\left(t\right)\in A_{P}$ for all
$f\in F$. }\\
{\large{}Define $\alpha\left(F\right)=a$ and $H\left(F\right)=L$. }{\large\par}

{\large{}Then as $\min L>m$ we have the first hypothesis of induction
is satiesfied. And if $n=1,$ we have }{\large\par}

{\large{}$a+\sum_{t\in L}f\left(t\right)\in A_{P}\subseteq A_{N}$.
And if $n>1$,
\[
\sum_{i=1}^{n}\left(\alpha\left(G_{i}\right)+\sum_{t\in H\left(G_{i}\right)}f_{i}\left(t\right)\right)+\left(a+\sum_{t\in L}f\left(t\right)\right)\in A_{N}
\]
for all $G_{1}\subsetneq G_{2}\subsetneq.....\subsetneq G_{n}=F$.}{\large\par}

{\large{}This completes the inductive proof of central set theorem.}{\large\par}
\end{proof}

\section{\textbf{\large{}central set theorem in non commutative semigroup}}

{\large{}In this section we will give a combinatorial proof of the
non commutative extension of central set theorem. The non commutative
central set theorem was first established in \cite{key-2}. A version
of this paper was established in \cite{key-7}. To prove that result
we first need that a piecewise syndetic set is a $J$- set. We first
have to prove piecewise syndetic sets are $J$-sets in non commutative
settings. }{\large\par}
\begin{thm}
{\large{}For a semigroup $(S,\cdot)$, let $A\subseteq S$ be a left
piecewise syndetic set then $A$ is $J$- set.}{\large\par}
\end{thm}

\begin{proof}
{\large{}Let $A\subseteq S$ be a left piecewise syndetic set, then
there exists a finite set $F$ such that $\bigcup_{x\in F}x^{-1}A$
is left thick.}{\large\par}

{\large{}Take any $G\in\mathcal{P}_{f}\left(S^{\mathbb{N}}\right)$
and $G=\left\{ f_{1}f_{2},\ldots,f_{k}\right\} $. Let, $\mid F\mid=r$
and take the Hales-Jewett number $N=N(k,r)$ guranteed by Lemma 6.}{\large\par}

{\large{}Now, take a correspondence map sending each $\left(i_{1},i_{2},\ldots,i_{N}\right)\in[k]^{N}$
to the element $f_{i_{1}}\left(1\right)\cdot f_{i_{2}}\left(2\right)\cdots f_{i_{N}}\left(N\right)$. }{\large\par}

{\large{}Let, $H=\left\{ f_{i_{1}}\left(1\right)\cdot f_{i_{2}}\left(2\right)\cdots f_{i_{N}}\left(N\right):\left(f_{i_{1}},f_{i_{2}},\ldots,f_{i_{N}}\right)\in G^{N}\right\} $
be a finite set and note that each element in $H$ is correspond to
at least one element in $[k]^{N}$. }{\large\par}

{\large{}Now choose an element $x\in S$ such that $H\cdot x\subset\bigcup_{x\in F}x^{-1}A$.
Thus the set $H\cdot x$ is finitely colored into $\mid F\mid=r$
color. }{\large\par}

{\large{}Give a coloring of $[k]^{N}$ say $\chi$ such that, 
\[
\chi\left(i_{1},i_{2},\ldots,i_{N}\right)=\chi\left(f_{i_{1}}\left(1\right)\cdot f_{i_{2}}\left(2\right)\cdots f_{i_{N}}\left(N\right)\cdot x\right)
\]
Then there is a monochromatic combinatorial line in $[k]^{N}$ and
this corresponds a monochromatic configuration of the form, 
\[
\left\{ a_{1}\cdot f\left(t_{1}\right)\cdot a_{2}\cdot f\left(t_{2}\right)\cdots a_{n}\cdot f\left(t_{n}\right)\cdot a_{n+1}:f\in G\right\} \subset x^{-1}A
\]
for some $x\in F$. Where $F_{1}=\left\{ t_{1},t_{2},\ldots,t_{n}\right\} $
is the position of the wildcart set.}{\large\par}

{\large{}So, 
\[
\left\{ x\cdot a_{1}\cdot f\left(t_{1}\right)\cdot a_{2}\cdot f\left(t_{2}\right)\cdots a_{n}\cdot f\left(t_{n}\right)\cdot a_{n+1}:f\in G\right\} \subset A
\]
and letting, $x\cdot a_{1}=a\left(1\right)$ and $a_{i}=a\left(i\right)$
for all $i\in\left\{ 2,3,\ldots n+1\right\} $ and $t=\left(t_{1},t_{2},\ldots t_{n}\right)$
we have the required result.}{\large\par}
\end{proof}
\begin{thm}
{\large{}For a semigroup $(S,\cdot)$, let $A\subseteq S$ be a central
set then, there exists $m:\mathcal{P}_{f}\left(S^{\mathbb{N}}\right)\rightarrow\mathbb{N}$,
$\alpha\in\underset{F\in\mathcal{P}_{f}\left(S^{\mathbb{N}}\right)}{\times}S^{m\left(F\right)+1}$
and $\tau\in\underset{F\in\mathcal{P}_{f}\left(S^{\mathbb{N}}\right)}{\times}\mathcal{J}_{m\left(F\right)}$
such that}{\large\par}
\begin{enumerate}
\item {\large{}If, $F,G\in\mathcal{P}_{f}\left(S^{\mathbb{N}}\right)$ and
$F\subset G$, then $\tau\left(F\right)\left(m\left(F\right)\right)<\tau\left(G\right)\left(1\right)$
and }{\large\par}
\item {\large{}Whenever $n\in\mathbb{N}$, $G_{1},G_{2},\ldots,G_{n}\in\mathcal{P}_{f}\left(S^{\mathbb{N}}\right)$,
$G_{1}\subset G_{2}\subset\ldots\subset G_{n}$ and for each $i\in\left\{ 1,2,\ldots,n\right\} $,
$f_{i}\in G_{i}$, one has 
\[
\prod_{i=1}^{n}x\left(m\left(G_{i}\right),\alpha\left(G_{i}\right),\tau\left(G_{i}\right),f_{i}\right)\in A
\]
}{\large\par}
\end{enumerate}
\end{thm}

\begin{proof}
{\large{}Let $A$ be central set in $\left(S,\cdot\right)$ and then
from theorem \ref{Charac. central set} there exists a chain of piecewise
syndetic set 
\[
A\supseteq A_{1}\supseteq A_{2}\supseteq\cdots\supseteq A_{n}\supseteq\cdots
\]
 satisfying property \ref{1.1}. Let us fixed any $N\in\mathbb{N}$
and take the pieecewise syndetic set $A_{N}$.}{\large\par}

{\large{}We define $m\left(F\right)$, $\alpha\left(F\right)$ and
$\tau\left(F\right)$ for $F\in\mathcal{P}_{f}\left(\mathbb{N^{S}}\right)$
by induction on $\left|F\right|$ so that }{\large\par}

{\large{}(1) if $\emptyset\neq G\subsetneq F$, then $\tau\left(G\right)\left(m\left(G\right)\right)<\tau\left(F\right)\left(1\right)$
and }{\large\par}

{\large{}(2) whenever $n\in\mathbb{N}$, $\emptyset\ne G_{1}\subsetneq G_{2}\subsetneq\ldots\subsetneq G_{n}=F$
and for each $i\in\left\{ 1,2,\ldots,n\right\} $, $f_{i}\in G_{i}$,
then $\prod_{i=1}^{n}x\left(m\left(G_{i}\right),\alpha\left(G_{i}\right),\tau\left(G_{i}\right),f_{i}\right)\in A_{N}$. }{\large\par}

{\large{}Assume first that $F=\left\{ f\right\} $. Then $A_{N}$
is a $J$-set so pick $m\left(F\right)\in\mathbb{N}$, $\alpha\left(F\right)\in S^{m\left(F\right)+1}$
and $\tau\left(F\right)\in\mathscr{J}_{m\left(F\right)}$ such that
$x\left(m\left(F\right),\alpha\left(F\right),H\left(F\right),f\right)\in A_{N}$.}{\large\par}

{\large{}Now assume that $\left|F\right|>1$ and that $m\left(G\right)$,
$\alpha\left(G\right)$ and $H\left(G\right)$ have been defined for
all proper subsets $G$ of $F$. Let $k=\max\left\{ \tau\left(G\right)\left(m\left(G\right)\right):\emptyset\neq G\subsetneq F\right\} $.
Let }{\large\par}

{\large{}
\begin{align*}
M & =\{\prod_{i=1}^{n}x\left(m\left(G_{i}\right),\alpha\left(G_{i}\right),\tau\left(G_{i}\right),f_{i}\right):n\in\mathbb{N},\emptyset\ne G_{1}\subsetneq G_{2}\subsetneq\ldots\subsetneq G_{n}\subsetneq F.\\
 & \text{ and }f_{i}\in G_{i}\text{ for each }i\in\left\{ 1,2,\ldots,n\right\} \}
\end{align*}
}{\large\par}

{\large{}Let $B=A_{N}\cap\bigcap_{b\in M}b^{-1}A_{N}$. Since $M$
is a finite subset of $A_{N}$, $B\supseteq A_{M}$ for some $M\in\mathbb{N}$
and therefore $B$ is a $J$-set. Since $A_{M}$ is a $J$-set. Pick
by Lemma \cite[14.14.3]{key-6} $m\left(F\right)\in\mathbb{N}$, $\alpha\left(F\right)\in S^{m\left(F\right)+1}$
and $\tau\left(F\right)\in\mathcal{J}_{m\left(F\right)}$ such that
$\tau\left(F\right)\left(1\right)>k$ and for each $f\in F$, $x\left(m\left(F\right),\alpha\left(F\right),\tau\left(F\right),f\right)\in A_{M}$.}{\large\par}

{\large{}Hypothesis (1) is satisfied directly. To verify hypothesis
(2), let $n\in\mathbb{N},\emptyset\ne G_{1}\subsetneq G_{2}\subsetneq\ldots\subsetneq G_{n}=F$
and for each $i\in\left\{ 1,2,\ldots,n\right\} $, let $f_{i}\in G_{i}$.
If $n=1$, then $x\left(m\left(G_{1}\right),\alpha\left(G_{1}\right),\tau\left(G_{1}\right),f_{i}\right)\in A_{M}$,
so assume that $n>1$. Let $b=\prod_{i=1}^{n-1}x\left(m\left(G_{i}\right),\alpha\left(G_{i}\right),\tau\left(G_{i}\right),f_{i}\right)$.
Then $b\in A_{M}$ so $x\left(m\left(G_{n}\right),\alpha\left(G_{n}\right),\tau\left(G_{n}\right),f_{i}\right)\in b^{-1}A_{M}$
so $\prod_{i=1}^{n}x\left(m\left(G_{i}\right),\alpha\left(G_{i}\right),\tau\left(G_{i}\right),f_{i}\right)\in A_{M}$
as required.}{\large\par}
\end{proof}
\textbf{\large{}Acknowledges:}{\large{} The first author acknowledges
the grant UGC-NET SRF fellowship with id no. 421333 of CSIR-UGC NET
December 2016.}{\large\par}

\end{document}